\documentclass[a4paper,oneside]{amsart}
\usepackage[T1]{fontenc}
\usepackage[utf8]{inputenc}
\usepackage[english]{babel}
\usepackage[autostyle=true]{csquotes}
\usepackage{latexsym, amssymb, amsmath}
\usepackage{amsfonts}
\usepackage{amsthm}
\usepackage{mathabx}
\makeatletter
\renewcommand*\env@matrix[1][*\c@MaxMatrixCols c]{%
	\hskip -\arraycolsep
	\let\@ifnextchar\new@ifnextchar
	\array{#1}}
\makeatother
\usepackage{url}
\usepackage{float}
\usepackage{caption}
\usepackage{subcaption}
\usepackage{xcolor}
\usepackage{bbm}
\usepackage[textsize=scriptsize]{todonotes}
\usepackage[shortlabels]{enumitem}
\usepackage{thmtools,mathtools}
\usepackage{nameref,hyperref,cleveref}
\hypersetup{%
    colorlinks=true, linktocpage=true, pdfstartpage=3, pdfstartview=FitV,%
    breaklinks=true, pdfpagemode=UseNone, pageanchor=true, pdfpagemode=UseOutlines,%
    plainpages=false, bookmarksnumbered, bookmarksopen=true, bookmarksopenlevel=1,%
    hypertexnames=true, pdfhighlight=/O,
    urlcolor=webbrown, linkcolor=RoyalBlue, citecolor=webgreen, 
    pdfsubject={},%
    pdfkeywords={},%
    pdfcreator={pdfLaTeX},%
}   
\definecolor{webgreen}{rgb}{0,.5,0}
\definecolor{webbrown}{rgb}{.5,.5,.5}
\definecolor{RoyalBlue}{rgb}{0,0,.5}
\usepackage[displaymath, mathlines,modulo,pagewise]{lineno}
\usepackage{hyperref}
\usepackage[ruled,vlined]{algorithm2e}
\usepackage{tikz-cd}
\usepackage{graphicx}
\graphicspath{ {./images/} }
\usepackage{mathtools}

\newcommand{\bF}{\mathbb{F}}

\newcommand{\cC}{\mathcal{C}}
\newcommand{\cV}{\mathcal{V}}

\newcommand{\mg}{\mathfrak{g}}
\newcommand{\mh}{\mathfrak{h}}
\newcommand{\Ass}{\mathbf{AAlg}}
\newcommand{\LieAlg}{\mathbf{LieAlg}}
\newcommand{\LeibAlg}{\mathbf{LeibAlg}}
\newcommand{\Pois}{\mathbf{PoisAlg}}
\newcommand{\NAlg}{\mathbf{NAlg}}
\newcommand{\CPoisAlg}{\mathbf{CPoisAlg}}
\newcommand{\CAAlg}{\mathbf{CAAlg}}
\newcommand{\Grp}{\mathbf{Grp}}

\newcommand{\id}{\operatorname{id}}
\newcommand{\Hom}{\operatorname{Hom}}
\newcommand{\Bim}{\operatorname{Bim}}
\newcommand{\SplExt}{\operatorname{SplExt}}
\newcommand{\Imm}{\operatorname{Im}}
\newcommand{\End}{\operatorname{End}}
\newcommand{\Der}{\operatorname{Der}}
\newcommand{\Bider}{\operatorname{Bider}}
\newcommand{\Act}{\operatorname{Act}}
\newcommand{\Aut}{\operatorname{Aut}}

\newcommand{\Inn}{\operatorname{Inn}}

\newtheorem{thm}{Theorem}[section]

\newtheorem{corollary}[thm]{Corollary}
\newtheorem{definition}[thm]{Definition}
\newtheorem{rem}[thm]{Remark}
\newtheorem{prop}[thm]{Proposition}
\newtheorem{ex}[thm]{Example}

\makeatletter
\g@addto@macro{\endabstract}{\@setabstract}
\newcommand{\authorfootnotes}{\renewcommand\thefootnote{\@fnsymbol\c@footnote}}%
\makeatother

\date{}

\usepackage[%
backend=bibtex,%
style=numeric-comp,%
sorting=nyt,%
hyperref,%
maxnames=50,
]{biblatex}

\AtBeginBibliography{}

\renewbibmacro{in:}{%
	\ifentrytype{article}
	{}
	{\bibstring{in}%
		\printunit{\intitlepunct}}}

\makeatletter
\@namedef{subjclassname@2020}{\textup{2020} Mathematics Subject Classification}
\makeatother

\begin{document}
	\sloppy
	\begin{center}
	  \LARGE 
	  \title[On the representability of actions of Leibniz and Poisson algebras]{On the representability of actions of \\ Leibniz algebras and Poisson algebras} \par \bigskip
	  
	  \keywords{Action representable category, split extension, associative algebra, Leibniz algebra, Poisson algebra. \\ This work is supported by University of Milan, University of Palermo, University of Turin and by the “National Group for Algebraic and Geometric Structures, and their Applications” (GNSAGA – INdAM). The second author is also a postdoctoral researcher of the Fonds de la Recherche Scientifique--FNRS}
	  
	  \maketitle
		
		\normalsize
		\authorfootnotes
		ALAN S.\ CIGOLI \textsuperscript{1}, MANUEL MANCINI 
		\textsuperscript{2,3} 
		AND
		GIUSEPPE METERE \textsuperscript{4} \par \bigskip 
		
		\textsuperscript{1} Dipartimento di Matematica ``Giuseppe Peano'', Università degli Studi di Torino, via Carlo Alberto 10, 10123 Torino, Italy. \\ \email{alan.cigoli@unito.it, ORCID: 0000-0002-5181-5096.}\par \bigskip
		
		\textsuperscript{2} Dipartimento di Matematica e Informatica, Università degli Studi di Palermo, via Archirafi 34, 90123 Palermo, Italy. \\ \email{manuel.mancini@unipa.it, ORCID: 0000-0003-2142-6193.}\par \bigskip
		
		\textsuperscript{3} Institut de Recherche en Mathématique et Physique, \\Université catholique de Louvain, \\chemin du cyclotron 2 bte L7.01.02, B--1348 Louvain-la-Neuve, Belgium. \\ \email{manuel.mancini@uclouvain.be}\par \bigskip
		
		\textsuperscript{4} Dipartimento di Scienze per gli Alimenti, la Nutrizione e l'Ambiente,\\Università degli Studi di Milano Statale,\\ via Luigi Mangiagalli 25, 20133 Milano, Italy.\\
		\email{giuseppe.metere@unimi.it, ORCID: 0000-0003-1839-3626.}\par \bigskip
	\end{center}

\begin{abstract}
	In a recent paper, motivated by the study of central extensions of associative algebras, G.~Janelidze introduces the notion of  weakly action representable category. In this paper, we show that the category of Leibniz algebras is weakly action representable and we characterize the class of acting morphisms. Moreover, we study the representability of actions of the category of Poisson algebras by describing explicitly a universal strict general actor.
\end{abstract}

\bigskip

\section*{Introduction}

\bigskip

Internal object actions were defined in \cite{IntAct} by F.~Borceux, G.~Janelidze and G.~M.~Kelly in order to recapture categorically several algebraic notions of action, such as the action of a group $G$ on another group $H$, the action of a Lie algebra $\mg$ on another Lie algebra $\mh$ and so on. In the same paper, the authors introduced the notion of \emph{representable action}: an object $X$ has representable actions if the functor $\Act(-,X)$, sending each object $B$ to the set of actions of $B$ on $X$, is representable (see \Cref{sec:Preliminaries} for further details). In \cite{ActRepr}, action representability was extensively studied in the semi-abelian context and it was proved that, for example, the category of commutative associative algebras over a field is not action representable.

In \cite{act_accessible} D.~Bourn and G.~Janelidze introduced the weaker notion of \emph{action accessible} category in order to include relevant examples that do not fit in the frame of action representable categories (such as rings, associative algebras and Leibniz algebras amongst others). A.~Montoli proved in \cite{Montoli} that all \emph{categories of interest} in the sense of G.~Orzech \cite{Orzech} are action accessible. On the other hand, the paper \cite{Casas} by J.~M.~Casas, T.~Datuashvili and M.~Ladra showed that a weaker notion of actor (namely, the universal strict general actor, USGA for short) is available for any category of interest $\cC$. 

Recently, G.~Janelidze introduced in \cite{WAR} the notion of \emph{weakly representable} action: for an object $X$ in a semi-abelian category $\cC$, a weak representation of the functor $\Act(-,X)$ is a pair $(T,\tau)$, where $T$ is an object of $\cC$ and  $\tau\colon\Act(-,X)\rightarrowtail\Hom_\cC(-,T)$ is a monomorphism of functors. When such monomorphism exists, one says that $X$ has weakly representable actions and $T$ is a \emph{weak actor} of $X$. In particular, when $\cC$ is a category of interest and $\operatorname{USGA}(X)$ is an object of $\cC$, then $\Act(-,X)$ has a weak representation (see \Cref{USGA2}).  

A semi-abelian category $\cC$ is said to be \emph{weakly action representable} if every object $X$ in $\cC$ has a weak representation of actions. This is true, for instance, for the category $\mathbf{AAlg}_{\bF}$ of associative algebras over a field $\bF$ \cite{WAR}. Notice that a category of interest needs not necessarily be weakly action representable, as observed by J.~R.~A.~Gray in \cite{Gray}.
However, thanks to the results of \cite{Casas}, we get that, for every object $X$ in a category of interest $\cC$, there exists a monomorphism of functors $\Act(-,X)\rightarrowtail\Hom_{\cC_G}(U(-)),\operatorname{USGA}(X))$, where $\cC_G$ is a suitable category containing $\cC$ as a full subcategory (see \Cref{USGA}) and $U \colon \cC \rightarrow \cC_G$ denotes the forgetful functor.

We analyze in details two specific cases: the category $\mathbf{LeibAlg}_{\bF}$ of Leibniz algebras (\Cref{sec:Leibniz}) and the category $\mathbf{PoisAlg}_{\bF}$ of Poisson algebras (\Cref{sec:Poisson}), where $\bF$ is a fixed field with $\operatorname{char}(\bF)\neq 2$. We show that the first one is a weakly action representable category and we provide a complete description of \emph{acting morphisms}, i.e.\ morphisms into a weak actor corresponding to internal actions, in this case and for associative algebras. Moreover, we study the representability of actions of the category $\Pois_{\bF}$ by describing explicitly a universal strict general actor $[V]=\operatorname{USGA}(V)$, for any Poisson algebra $V$, and the corresponding monomorphism of functors  
$$
\tau \colon \Act(-,V)\rightarrowtail\Hom_{\NAlg_{\bF}^2}(U(-),[V]),
$$
where $\NAlg_{\bF}^2$ is the category of algebras over $\bF$ with two not necessarily associative bilinear operations and $U \colon \Pois_{\bF} \rightarrow \NAlg_{\bF}^2$ denotes the forgetful functor. Finally, we extend this study to the subvariety $\CPoisAlg_{\bF}$ of commutative Poisson algebras and we prove that there exists a natural isomorphism
$$
\Act(-,V) \cong \Hom_{\NAlg_{\bF}^2}(\tilde{U}(-),[V]_c),
$$
for any commutative Poisson algebra $V$, where $\tilde{U} \colon \CPoisAlg_{\bF} \rightarrow \NAlg_{\bF}^2$ denotes the forgetful functor and $[V]_c$ is the universal strict general actor of $V$ in the category $\CPoisAlg_{\bF}$.

\bigskip

\section{Preliminaries}\label{sec:Preliminaries}

\bigskip

The notion of semi-abelian category was introduced in \cite{Semi-Ab} by G.~Janelidze, L.~Márki and W.~Tholen in order to provide a categorical setting which would capture  algebraic properties of groups, rings and algebras. Let us recall that a category~$\cC$ is semi-abelian when it is finitely complete, Barr-exact, pointed, protomodular and has finite coproducts. 

One notion which is central in the present article is that of split extension. Let~$X,B$ be objects of a semi-abelian category $\cC$; a \emph{split extension} of $B$ by $X$ is a diagram

\begin{equation}\label{eq:split_ext}
\begin{tikzcd}
0\ar[r]
&X\arrow [r, "k"]
&A \arrow[r, shift left, "\alpha"] & 
B\ar[r]\ar[l, shift left, "\beta"]
&0 
\end{tikzcd}	
\end{equation}
in $\cC$ such that $\alpha \circ \beta = \id_B$ and $(X,k)$ is a kernel of $\alpha$. Notice that protomodularity implies that the pair $(k,\beta)$ is jointly strongly epic, $\alpha$ is indeed the cokernel of $k$ and diagram (\ref{eq:split_ext}) represents an extension of $B$ by $X$ in the usual sense. Morphisms of split extensions are morphisms of extensions that commute  with the sections. Let us observe that, again by protomodularity, a morphism of split extensions fixing $X$ and $B$ is necessarily an isomorphism.
For  an object $X$ of $\cC$, we  define the functor
\[
\SplExt(-,X)\colon \cC^{\operatorname{op}} \rightarrow \textbf{Set}
\]
which assigns to any object $B$ of $\cC$, the set $\SplExt(B,X)$ of isomorphism classes of split extensions of $B$ by $X$, and to any arrow $f\colon B'\to B$ the \emph{change of base } function  $f^*\colon \SplExt(B,X) \to \SplExt(B',X)$ given by pulling back along $f$.

A feature of semi-abelian categories is that one can define a notion of internal action. If we fix an object $X$, actions on $X$ give rise to a functor 
\[
\Act(-,X)\colon \cC^{\operatorname{op}} \rightarrow \textbf{Set}.
\]
In fact, we will not describe explicitly internal actions, since there is a natural isomorphism of functors $\Act(-,X)\cong \SplExt(-,X)$, and split extensions are more handy to work with (we refer the interested reader to \cite{ActRepr}, where this isomorphism is described in detail). This justifies the terminology in the definition that follows. 
\begin{definition}
A semi-abelian category $\cC$ is \emph{action representable} if for every object $X$ in $\cC$, the functor $\SplExt(-,X)$ is representable. This means that there exists an object $[X]$ of $\cC$, called the \emph{actor} of $X$, and a natural isomorphism
$$
\SplExt(-,X) \cong \Hom_{\cC}(-,[X]).
$$
\end{definition}
The prototype examples of action representable categories are the category $\Grp$ of groups and the category $\LieAlg_{R}$ of  Lie algebras over a commutative ring $R$. In the first case, it is well known that every split extension of $B$ by $X$ is represented by a homomorphism $B \rightarrow \Aut(X)$, where the actor $\Aut(X)$  of $X$ is the group of automorphisms of the group $X$. In the case of Lie algebras, a split extension of $B$ by $X$ is represented by a homomorphism $B \rightarrow \Der(X)$, where $\Der(X)$ is the Lie algebra of derivations of $X$. Therefore, $\Der(X)$ is the actor of $X$.

However the notion of action representable category has proven to be quite restrictive. For instance, in \cite{Tim} the authors proved that, if a variety $\cV$ of non-associative algebras (over an infinite field $\bF$ with $\operatorname{char}(\bF)\neq 2$) is action representable, then $\cV=\LieAlg_{\bF}$.

In \cite{WAR} G.\ Janelidze introduced a weaker notion for the representability of actions in a semi-abelian category $\cC$.

\begin{definition}
	A semi-abelian category $\cC$ is \emph{weakly action representable} if for every objext $X$ in $\cC$, the functor $\SplExt(-,X)$ admits a weak representation. This means that  there exist an object $T$ of  $\cC$ and a monomorphism  of functors
	\[
	\tau\colon \SplExt(-,X) \rightarrowtail \Hom_{\cC}(-,T).
	\]
	An object $T$ as above is called \emph{weak actor} of $X$;  a morphism $\varphi\colon B \rightarrow T \in \Imm(\tau_B)$ is called \emph{acting morphism}.
\end{definition}

Notice that every action representable category $\cC$ is weakly action representable. In this case, $T=[X]$ is the actor of $X$, $\tau$ is a natural isomorphism and every arrow $\varphi\colon B \rightarrow [X]$ is an acting morphism.

\subsection{Associative Algebras}
The case of associative algebras over a field $\bF$ is studied in \cite{WAR}: the category $\Ass_{\bF}$ of associative algebras over $\bF$ is weakly action representable.  Let us recall the basic constructions.  

Given an associative algebra $X$, a weak actor of $X$ is the associative algebra
\[
\Bim(X)=\lbrace (f*-,-*f) \in \End(X)\times \End(X)^{\text{op}}\; \vert\cdots\qquad\qquad\qquad\qquad\qquad\qquad
\]
\[
\qquad\qquad\qquad  \cdots\vert \;f*(xy)=(f*x)y, (xy)*f=x(y*f), x(f*y)=(x*f)y, \; \forall x,y \in X \rbrace
\]
of \emph{bimultipliers} of $X$ (see \cite{MacLane58}, where they are called \emph{bimultiplications}). Moreover, the isomorphism classes of split extensions of an associative algebra $B$ by $X$ are in bijection with the class of morphisms
\[
B \rightarrow \Bim(X), \; \;  a \mapsto (a*-,-*a),\; \; \forall a \in B,
\]
which satisfy the condition
\begin{equation}\label{eq:acting_morph}
a*(x*b)=(a*x)*b, \; \; \forall a,b \in B, \; \forall x \in X,
\end{equation}
i.e.\ the left multiplier $a * -$ and the right multiplier $-*b$ are permutable.  Notice that $(a*-,-*a)$ can be considered  respectively the left and the right components of the action of $a \in B$ on $X$.

\Cref{eq:acting_morph} can be used to characterize the class of acting morphisms in the category $\Ass_{\bF}$.\ In \cite{MacLane58} S.\ Mac Lane described, for a ring $\Lambda$, the $\Lambda-$bimodule structures over an abelian group $K$ in terms of ring morphisms from $\Lambda$ to the ring of bimultipliers of $K$. The following is a straightforward generalization to actions on an object which is not necessarily abelian. 

\begin{prop}
	Let $B$ and $X$ be associative algebras over $\bF$ and let 
\[
\varphi \in \Hom_{\Ass_{\bF}}(B,\Bim(X))
\]
defined by
\[
\varphi(a)=(a*_{\varphi}-,-*_{\varphi}a), \; \; \forall a \in B.
\]
	Then $\varphi$ is an \emph{acting morphism} if and only if
	\[
	a*_{\varphi}(x*_{\varphi}b)=(a*_{\varphi}x)*_{\varphi}b,
	\]
    for every $a,b \in B$ and for every $x \in X$.
\end{prop}

\begin{proof}
 We recall from \cite{WAR} that a weak representation of an associative algebra $X$ is given by a pair $(\Bim(X),\tau)$, where 
 \[
 \tau\colon \SplExt(-,X) \rightarrowtail \Hom_{\Ass_{\bF}}(-,\Bim(X))
 \]
 is the monomorphism of functors which associate with any split extension $A$ of $B$ by $X$, as in diagram (\ref{eq:split_ext}), the morphism $\varphi \colon B \rightarrow \Bim(X)$ defined by
 \[
 \varphi(a) = (a*_{\varphi}-,-*_{\varphi}a)=(\beta(a) \cdot_A -, - \cdot_A \beta(a)),
 \]
 for every $a \in B$. It follows from the associativity of the algebra $A$ that the left multiplier $a*_{\varphi}-$ and the right multiplier $-*_{\varphi}b$ are permutable, for every $a,b \in B$. Conversely, with any morphism  $\varphi \colon B \rightarrow \Bim(X)$ satisfying 
 \[
 a*_{\varphi}(x*_{\varphi}b)=(a*_{\varphi}x)*_{\varphi}b, \; \; \forall a,b \in B, \; \forall x \in X,
 \]
we can associative the split extension of $B$ by $X$ given by the semi-direct product $B \ltimes X$, as in the proof of \cite[Proposition 2.1]{ActRepr}, i.e.\  $\varphi \in \Imm(\tau_B)$.

\end{proof}

\subsection{Jordan Algebras}

An example of variety of non-associative algebras over a field $\bF$ which is not a \emph{weakly action representable category} is given by \emph{Jordan algebras}. Recall that a \emph{Jordan algebra} over a field $\bF$ is a non-associative commutative algebra $(J, \cdot)$ over $\bF$ which satisfies the \emph{Jordan identity}
\[
(xy)(xx)=x(y(xx)), \; \; \forall x,y \in J.
\]

In \cite{WAR} G.\ Janeldize showed that every weakly action representable category is action accessible (see \cite{act_accessible}). In fact the variety $\textbf{JordAlg}_{\bF}$ of Jordan algebras over $\bF$ is not action accessible (see \cite{Cigoli}), hence it is not weakly action representable.

\bigskip

\section{Leibniz Algebras}\label{sec:Leibniz}

\bigskip

We assume that $\bF$ is a field with $\operatorname{char}(\bF)\neq2$.

\begin{definition}[\cite{loday1993version}]
	A \emph{(right) Leibniz algebra} over $\mathbb{F}$ is a vector space $\mg$ over $\mathbb{F}$ endowed with a bilinear map (called $commutator$ or $bracket$) $\left[-,-\right]\colon \mg\times \mg \rightarrow \mg$ which satisfies the \emph{(right) Leibniz identity}
\[
	\left[\left[x,y\right],z\right]=\left[[x,z],y\right]+\left[x,\left[y,z\right]\right], \;\;\forall x,y,z\in \mg.
\]
	
\end{definition}
	
	Every Lie algebra is a Leibniz algebra and every Leibniz algebra with skew-symmetric commutator is a Lie algebra. In fact, the full inclusion $i\colon \LieAlg_{\bF}\rightarrow \LeibAlg_{\bF}$ has a left adjoint $\pi\colon \LeibAlg_{\bF} \rightarrow \LieAlg_{\bF}$ that associates, with every Leibniz algebra $\mg$, its quotient $\mg/\mg^{\operatorname{ann}}$, where $\mg^{\operatorname{ann}}=\langle \left[x,x\right] |\,x\in \mg \rangle$ is the \emph{Leibniz kernel} of $\mg$. Note that  $\mg^{\operatorname{ann}}$ is an abelian algebra. 
	
	We define the left and the right center of a Leibniz algebra
\[
	\operatorname{Z}_l(\mg)=\left\{x\in \mg \,|\,\left[x,\mg\right]=0\right\},\,\,\, \operatorname{Z}_r(\mg)=\left\{x\in \mg \,|\,\left[\mg,x\right]=0\right\}
\]
	and we observe that they coincide when $\mg$ is a Lie algebra. The \emph{center} of $\mg$ is $\operatorname{Z}(\mg)=\operatorname{Z}_l(\mg)\cap \operatorname{Z}_r(\mg)$. In general $\operatorname{Z}_r(\mg)$ is an ideal of
	$\mg$, while the left center may not even be a subalgebra.
	
\subsection{Derivations and Biderivations}
	
	The definition of \emph{derivation} is the same as in the case of Lie algebras.
	\begin{definition}
		Let $\mg$ be a Leibniz algebra over $\bF$. A \emph{derivation} of $\mg$ is a linear map $d\colon \mg \rightarrow \mg$ such that 
		\[
		d(\left[x,y\right])=\left[d(x),y\right]+\left[x,d(y)\right],\,\;\; \forall x,y\in \mg.
		\]
	\end{definition}
	
	The right multiplications of $\mg$ are particular derivations called \emph{inner derivations} and an equivalent way to define a Leibniz algebra is to say that the (right) adjoint map $\operatorname{ad}_x=\left[-,x\right]$ is a derivation, for every $x\in \mg$. On the other hand the left adjoint maps are not derivations in general. 
	
	With the usual bracket \hbox{$\left[d_1,d_2\right]=d_1\circ d_2 - d_2\circ d_1$}, the set $\Der(\mg)$ is a Lie algebra and the set  $\Inn(\mg)$ of all inner derivations of $\mg$ is an ideal of $\Der(\mg)$. Furthermore, $\Aut(\mg)$ is a Lie group and the associated Lie algebra is $\Der(\mg)$.
	
	The definitions of \emph{anti-derivation} and \emph{biderivation} for a Leibniz algebra were introduced by J.-L. Loday in \cite{loday1993version}.
	
	\begin{definition}
	An \emph{anti-derivation} of a Leibniz algebra $\mg$ is a linear map $D\colon \mg \rightarrow \mg$ such that
\[
	D([x,y])=[D(x),y]-[D(y),x], \; \; \forall x,y \in \mg.
\]
	\end{definition}
   \noindent
 
 One can check that, for every $x \in \mg$, the left multiplication $\operatorname{Ad}_x=[x,-]$ defines and anti-derivation. We observe that in the case of Lie algebras, there is no difference between a derivation and an anti-derivation.
    
 \begin{rem}
    	The set of anti-derivations of a Leibniz algebra $\mg$ has a $\Der(\mg)$-module structure with the multiplication
 \[
    	d \cdot D \coloneq [D,d]= D \circ d - d \circ D,
 \]
    	for every $d \in \Der(\mg)$ and for every anti-derivation $D$.
 \end{rem}
 
 \begin{definition}
 	Let $\mg$ be a Leibniz algebra. A \emph{biderivation} of $\mg$ is a pair $(d,D)$ where $d$ is a derivation and $D$ is an anti-derivation, such that
 \begin{equation*}
 		[x,d(y)]=[x,D(y)], \; \; \forall x,y \in \mg.
 \end{equation*}
 \end{definition}

The set of all biderivations of $\mg$, denoted by $\Bider(\mg)$, has a Leibniz algebra structure with the bracket
\[
[(d,D),(d',D')]=(d \circ d' - d' \circ d, D \circ d' - d' \circ D), \; \; \forall (d,D),(d,D') \in \operatorname{Bider}(\mg)
\]
and it is possibile to define a Leibniz algebra morphism
\[
\mg \rightarrow \Bider(\mg)
\]
by
\[
x \mapsto (-\operatorname{ad}_x, \operatorname{Ad}_x), \; \; \forall x \in \mg.
\]
The pair $(-\operatorname{ad}_x, \operatorname{Ad}_x)$ is called \emph{inner biderivation} of $\mg$ and the set of all inner biderivations forms a Leibniz subalgebra of $\Bider(\mg)$. We refer the reader to \cite{MM} for a complete classification of the Leibniz algebras of biderivations of low-dimensional Leibniz algebras over a general field $\bF$ with $\operatorname{char}(\bF) \neq 2$.
	
\subsection{Split Extensions of Leibniz algebras}

By studying biderivations of a Leibniz algebra $\mh$, we can classify the split extensions with kernel $\mh$. This relies on the correspondence between actions and split extensions available in any semi-abelian category, as explained in \Cref{sec:Preliminaries}. Since the variety of Leibniz algebra is a category of interest (see \cite{Orzech}), it is convenient here to describe \emph{internal actions} in terms of the so-called \emph{derived actions}.
    
\begin{definition}\label{Rem:SplitLeib}
Let 
\begin{equation}\label{diag:SplLeib}
		\begin{tikzcd}
			0\ar[r]
			&\mh \arrow [r, "i"]
			&\hat{\mg} \arrow[r, shift left, "\pi"] & 
			\mg \ar[r]\ar[l, shift left, "s"]
			&0 
		\end{tikzcd}	
\end{equation}
	be a split extension of Leibniz algebras. The pair of bilinear maps 
		\[
	l\colon \mg \times \mh \rightarrow \mh, \; \; \; r\colon \mh \times \mg \rightarrow \mh
	\]
	defined by
	\[
	l_x(b)=[s(x),i(b)]_{\hat{\mg}}, \; \; r_y(a)=[i(a),s(y)]_{\hat{\mg}}, \; \; \forall x,y \in \mg, \; \forall a,b \in \mh,
	\]
	where $l_x=l(x,-)$ and $r_y=r(-,y)$, is called the \emph{derived action} of $\mg$ on $\mh$ associated with \eqref{diag:SplLeib}.
\end{definition}

Given a pair of bilinear maps 
\[
l\colon \mg \times \mh \rightarrow \mh, \; \; \; r\colon \mh \times \mg \rightarrow \mh,
\]
one can define a bilinear operation on the direct sum of vector spaces $\mg \oplus \mh$
\[
[(x,a),(y,b)]_{(l,r)}=([x,y]_{\mg},[a,b]_{\mh}+l_x(b)+r_y(a)), \; \; \forall (x,a),(y,b) \in \mg \oplus \mh.
\]
By Theorem 2.4 in \cite{Orzech}, this defines a Leibniz algebra structure on $\mg \oplus \mh$ if and only if the pair $(l,r)$ is a derived action of $\mg$ on $\mh$. This in turn is equivalent to a set of conditions on the pair $(l,r)$, as explained in the following proposition, which is a special case of Proposition 1.1 in \cite{Datuashvili}.

\begin{prop}\label{PropLeib}
$(\mg \oplus \mh, [-,-]_{(l,r)})$ is a Leibniz algebra if and only if
\begin{enumerate}
	\item[(L1)] $r_x([a,b])=[r_x(a),b]+[a,r_x(b)]$;
	\item[(L2)] $l_x([a,b])=[l_x(a),b]-[l_x(b),a]$;
	\item[(L3)] $[a,r_x(b)+l_x(b)]=0$;
	\item[(L4)] $r_{[x,y]}=[r_y,r_x]=r_y \circ r_x - r_x \circ r_y$;
	\item[(L5)] $l_{[x,y]}=[r_y,l_x]=r_y \circ l_x - l_x \circ r_y$;
	\item[(L6)] $l_x \circ (l_y + r_y)=0$;
\end{enumerate}
for every $x,y \in \mg$ and for every $a,b \in \mh$. The resulting Leibniz algebra is the \emph{semi-direct product} of $\mg$ and $\mh$ and it is denoted by $\mg \ltimes \mh$.
\end{prop} 

\begin{rem}
	Notice that, for any split extension \eqref{diag:SplLeib} and the corresponding derived action $(l,r)$, there is an isomorphism of Leibniz algebra split extensions
	\begin{equation*}\label{iso}
		\begin{tikzcd} 
			0\ar[r] 
			&\mh \arrow [r, "i_2"] \ar[d, "\id_{\mh}"'] 
			&\mg \ltimes \mh \arrow[r, shift left, "\pi_1"] \ar[d, "\theta"]& 
			\mg \ar[r]\ar[l, shift left, "i_1"]\ar[d, "\id_{\mg}"] 
			&0 \\
			0\ar[r] 
			&\mh \arrow [r, "i"] 
			&\hat{\mg} \arrow[r, shift left, "\pi"] & 
			\mg \ar[r]\ar[l, shift left, "s"] 
			&0 
		\end{tikzcd} 
	\end{equation*}
where $i_1,i_2,\pi_1$ are the canonical injections and projection and $\theta \colon \mg \ltimes \mh \rightarrow \hat{\mg}$ is defined by $\theta(x,a)=s(x)+i(a)$, for every $(x,a) \in \mg \oplus \mh$.
\end{rem}

\begin{rem}\label{semidirect}
The first three equations of \Cref{PropLeib} state that, for every $x \in \mg$, the pair
$$
(-r_x,l_x)
$$
is a biderivation of the Leibniz algebra $\mh$. Moreover, from the equalities (L4)-(L5), we have that the linear map
$$
\varphi\colon \mg \rightarrow \Bider(\mh)
$$
defined by
$$
\varphi(x)=(-r_x,l_x), \; \; \forall x \in \mg 
$$
is a Leibniz algebra morphism. Indeed
	$$
	\varphi([x,y]_{\mg})=(-r_{[x,y]_{\mg}},l_{[x,y]_{\mg}})=(-[r_y,r_x],[r_y,l_x])
	$$
	and
	$$
	[\varphi(x),\varphi(y)]_{\Bider(\mh)}=[(-r_x,l_x),(-r_y,l_y)]_{\Bider(\mh)}=([-r_x,-r_y],[l_x,-r_y])=
	$$
	$$
	=([r_x,r_y],-[l_x,r_y])=(-[r_y,r_x],[r_y,l_x]).
	$$
\end{rem}

On the other hand, given a Leibniz algebra morphism 
$$
\varphi\colon \mg \rightarrow \Bider(\mh)
$$
with notation
$$
\varphi(x)=([\![-,x]\!],[\![x,-]\!]), \; \; \forall x \in \mg,
$$
satisfying
$$
[\![x,[\![y,a]\!]-[\![a,y]\!]]\!]=0, \; \; \forall x,y \in \mg, \; \forall a \in \mh,
$$
we can associate the split extension
\begin{equation*}
	\begin{tikzcd}
		0\ar[r]
		&\mh \arrow [r, "i"]
		&(\mg \oplus \mh, [-,-]_{\varphi}) \arrow[r, shift left, "\pi"] & 
		\mg \ar[r]\ar[l, shift left, "s"]
		&0 
	\end{tikzcd}	
\end{equation*}
where the Leibniz algebra structure of $\mg \oplus \mh$ is given by
$$
	[(x,a),(y,b)]_{\varphi}=([x,y]_{\mg},[a,b]_{\mh}+[\![x,b]\!]-[\![a,y]\!]), \; \; \forall (x,a),(y,b) \in \mg \oplus \mh.
$$
However a generic morphism from $\mg$ to $\Bider(\mathfrak{h})$ needs not give rise to a split extension, as the following example shows. 

\begin{ex}[\cite{Metere}]
	Let $\mg=\bF$ be the abelian one-dimensional algebra. Then the morphism $\varphi\colon \bF \rightarrow \Bider(\bF)=\End(\bF)^2$ defined by
	$$
	\varphi(a)=(d_a,D_a),
	$$
	where
	$$
	d_a(x)=-ax, \; D_a(x)=ax, \; \; \forall a,x \in \bF
	$$ 
	does not define a split extension of $\bF$ by itself. Indeed in general
	$$
	D_a(D_b(x)-d_b(x))=a(bx-(-bx))=2abx \neq 0.
	$$
\end{ex}

\begin{ex} \textbf{The (bi-)adjoint extension}\\
	Let $\mg$ be a Leibniz algebra and consider the canonical action of $\mg$ on itself given by the pair of linear maps
	$$
	r_x=\operatorname{ad}_x=[-,x], \; \; \forall x \in \mg, 
	$$ 
	$$
	l_y=\operatorname{Ad}_y=[y,-], \; \; \forall y \in \mg.
	$$
	We have a split extension of $\mg$ by itself with associated morphism
	$$
	\mg \rightarrow \Bider(\mg)
	$$
	defined by
	$$
	x \rightarrow (-\operatorname{ad}_x,\operatorname{Ad}_x), \; \; \forall x \in \mg,
	$$
	which obviously satisfies the condition
	$$
	\operatorname{Ad}_x \circ (\operatorname{Ad}_y+\operatorname{ad}_y)=0, \; \; \forall x,y \in \mg.
	$$
	Indeed, for every $z \in \mg$
	$$
	[x,[y,z]+[z,y]]=[x,[y,z]]+[x,[z,y]]=
	$$
	$$
	=[[x,y],z]-[[x,z],y]+[[x,z],y]-[[x,y]z]=0
	$$
	Thus the Leibniz algebra morphism which defines the inner biderivations of $\mg$ is associated with the canonical (bi-)adjoint extension of $\mg$ by itself.
\end{ex}

\begin{ex}
	Let $\mh$ be a Leibniz algebra. It is well known that (see \cite{Casas} for more details), if $\mh$ has trivial center (i.e.\ $\operatorname{Z}(\mh)=0$) or if $\mh$ is \emph{perfect} (which means that $[\mh,\mh]=\mh$), then for every $(d,D),(d',D') \in \Bider(\mh)$ we have
	$$
	D(D'(x)-d'(x))=0, \; \; \forall x \in \mh.
	$$
	Thus, given any Leibniz algebra $\mg$, we can associate a split extension of $\mg$ by $\mh$ with any morphism
	$$
	\mg \rightarrow \Bider(\mh)
	$$
	and $\Bider(\mh)$ is the actor of $\mh$.
\end{ex}

\begin{rem}
	Let $\mg$ and $\mh$ be Lie algebras and let $\hat{\mg}$ be a Lie algebra split extension of $\mg$ by $\mh$. Then, as observed above, we have that
	$$
	\hat{\mg} \cong (\mg \oplus \mh, [-,-]_r),
	$$
	where the Lie bracket is defined by
	$$
	[(x,a),(y,b)]_r=([x,y]_{\mg},[a,b]_{\mh}-r_x(b)+r_y(a)), \; \; \forall (x,a),(y,b) \in \mg \oplus \mh.
	$$
	In this case the left component of the action of $\mg$ on $\mh$ is defined by 
	$$
	l_x(b)=-r_x(b), \; \; \forall x \in \mg, \; \forall b \in \mh,
	$$
	thus the equation (L6) is automatically satisfied and every morphism
	$$
	\mg \rightarrow \Bider(\mh), \; \; x \mapsto ([\![-,x]\!],[\![-,x]\!]), \; \; \forall x \in \mg
	$$
	represents a split extension of $\mg$ by $\mh$ in the category $\LieAlg_{\bF}$. Moreover the subalgebra $\lbrace (d,d) \; | \; d \in \Der(\mh) \rbrace$ of $\Bider(\mh)$ is a Lie algebra isomorphic to $\Der(\mh)$.
\end{rem}

We can now claim the following result.

\begin{thm}
Let $\mg$ and $\mh$ be Leibniz algebras over $\bF$.
\begin{itemize}
\item[(i)] The isomorphism classes of split extensions of $\mg$ by $\mh$ are in bijection with the Leibniz algebra morphisms
$$
\varphi\colon \mg \rightarrow \Bider(\mh), \; \; \varphi(x)=([\![-,x]\!],[\![x,-]\!]), \; \; \forall x \in \mg,
$$
which satisfy the condition
\begin{equation}\label{condition_Leibniz}
	[\![x,[\![y,a]\!]-[\![a,y]\!]]\!]=0, \; \; \forall x,y \in \mg, \; \forall a \in \mh.
\end{equation}

\item[(ii)] The category $\LeibAlg_{\bF}$ of Leibniz algebras over $\bF$ is weakly action representable.

\item[(iii)] A weak actor of an object $\mh$ in $\LeibAlg_{\bF}$ is the Leibniz algebra $\Bider(\mh)$.

\item[(iv)] $\varphi \in \Hom_{\LeibAlg_{\bF}}(\mg,\Bider(\mh))$ is an acting morphism if and only if it satisfies  condition $(\ref{condition_Leibniz})$.
\end{itemize}
\end{thm}

\begin{proof} {\ }\\
	\begin{itemize}
	\item[(i)] The first statement follows from  \Cref{semidirect}.
	\item[(ii)] Given any Leibniz algebra $\mh$, we take $T=\Bider(\mh)$ and we define $\tau$ in the following way: for every Leibniz algebra $\mg$, the component
	$$
	\tau_{\mg}\colon \SplExt(\mg,\mh) \rightarrow \Hom_{\LeibAlg_{\bF}}(\mg,\Bider(\mh))
	$$
	is the morphism in $\textbf{Set}$ which associates with any split extension 
	\begin{equation*}
		\begin{tikzcd}
			0\ar[r]
			&\mh \arrow [r, "i"]
			&\hat{\mg} \arrow[r, shift left, "\pi"] & 
			\mg \ar[r]\ar[l, shift left, "s"]
			&0 
		\end{tikzcd}	
	\end{equation*}
	the morphism $\varphi_{(l,r)}\colon \mg \rightarrow \Bider(\mh)$ defined by
	$$
	x \mapsto (-r_x,l_x), \; \; \forall x \in \mg
	$$
	(see \Cref{Rem:SplitLeib}). The transformation $\tau$ is natural. Indeed, for every Leibniz algebra morphism $f\colon \mg' \rightarrow \mg$, it is easy to check that the following diagram in $\textbf{Set}$
	$$
	\begin{tikzcd} 
		\SplExt(\mg,\mh) \arrow[r, "\tau_{\mg}"] \arrow[d, "{\SplExt(f,\mh)}"]
		& \Hom(\mg,\Bider(\mh)) \arrow[d,"{\Hom(f,\Bider(\mh))}"]\\ 
		\SplExt(\mg',\mh) \arrow[r, "\tau_{\mg'}"] & \Hom(\mg',\Bider(\mh)) 
	\end{tikzcd}
	$$
	is commutative.	Moreover, for every Leibniz algebra $\mg$, the morphism $\tau_{\mg}$ is an injection since every element of $\SplExt(\mg,\mh)$ is uniquely determined by the corresponding action of $\mg$ on $\mh$, i.e.\ by the pair of bilinear maps
	$$
	l \colon \mg \times \mh \rightarrow \mh, \; \; \;  r\colon \mh \times \mg \rightarrow \mh\,.
	$$
	Thus $\tau$ is a monomorphism of functors and the category $\LeibAlg_{\bF}$ is weakly action representable.
	\item[(iii)] It follows immediately from (ii) that a \emph{weak actor} of $\mh$ is the Leibniz algebra of biderivations $\Bider(\mh)$.
	\item[(iv)] Finally $\varphi \in \Hom_{\LeibAlg_{\bF}}(\mg,\Bider(\mh))$ is an acting morphism if and only if it defines a split extension of $\mg$ by $\mh$, i.e.\ if and only if 
	it satisfies the condition
	$$
	[\![x,[\![y,a]\!]-[\![a,y]\!]]\!]=0, \; \; \forall x,y \in \mg, \; \forall a \in \mh.
	$$
	\end{itemize}
\end{proof}

\bigskip

\section{Categories of Interest} \label{sec:Interest}

\bigskip

The result of the previous section can be viewed as a particular case of \Cref{USGA} below,  that is valid more in general for \emph{categories of interest}. 
In \cite{Casas} the authors studied the problem of representability of actions for a category of interest $\cC$. They introduced a corresponding category $\cC_G$ of objects satisfying a suitable smaller set of identities than $\cC$, so that $\cC$ becomes a subvariety of $\cC_G$. They proved that, for every object $X$ in $\cC$, there exists an object $\operatorname{USGA}(X)$ of $\cC_G$, called \emph{universal strict general actor} of $X$, with the following property: for every object $B$ in $\cC$ and for every action $\xi$ of $B$ on $X$, there exists a unique morphism $\varphi\colon B \rightarrow \operatorname{USGA}(X)$ in $\cC_G$ such that $\xi$ is uniquely determined by the action of $\varphi(B)$ on $X$. It was clear from their investigation that categories of interest are not action representable in general. In fact J.~R.~A.~Gray showed in \cite{Gray} that a category of interest may not even be weakly action representable. However, by the results in \cite{Casas}, we can deduce the following.

\begin{prop}\label{USGA}
	Let $\cC$ be a category of interest and let $X$ be an object of $\cC$. Then there exists a monomorphism of functors
	$$
	\tau \colon \Act(-,X) \rightarrowtail \Hom_{\cC_G}(U(-),\operatorname{USGA}(X)),
	$$
	where $U \colon \cC \rightarrow \cC_G$ denotes the forgetful functor. If moreover $\operatorname{USGA}(X)$ is an object of $\cC$, then the pair $(\operatorname{USGA}(X),\tau)$ is a weak representation of $\Act(-,X)$.
\end{prop}

\begin{proof}
	By the above discussion, for every object $B$ in $\cC$, there exists an injection
	$$
	\tau_B \colon \Act(B,X) \rightarrowtail \Hom_{\cC_G}(U(B),\operatorname{USGA}(X)).
	$$
	We want to prove that the collection $\lbrace \tau_B \rbrace_{B \in \cC}$ gives rise to a natural transformation $\tau$.
	
	Consider in $\cC$ a morphism $f\colon B' \rightarrow B$ and an action $\xi$ of $B$ on $X$. The naturality of $\tau$ is equivalent to saying that 
	$$
	\tau_{B'} (f^*(\xi))=(\tau_B(\xi)) \circ f,
	$$
	for every such $f$ and $\xi$, where $f^*=\Act(f,X)$. This follows immediately from Definition 3.6 of \cite{Casas}.
	
	Since $\cC$ is a full subcategory of $\cC_G$, when $\operatorname{USGA}(X)$ belongs to $\cC$, the pair $(\operatorname{USGA}(X),\tau)$ is a weak representation for the functor $\Act(-,X)$.	
\end{proof}

\begin{corollary}\label{USGA2}
	Let $\cC$ be a category of interest. If $\operatorname{USGA}(X)$ is an object of $\cC$ for every $X$ in $\cC$, then $\cC$ is a weakly action representable category.
\end{corollary}

In view of the last results, an explicit description of the USGA in concrete cases is very useful. Two examples were studied in \cite{Casas}:
\begin{itemize}
	\item the category $\textbf{AAlg}_{\bF}$, where $\operatorname{USGA}(X)=\Bim(X)$, for every associative algebra $X$;
	\item the category $\textbf{LeibAlg}_{\bF}$, where $\operatorname{USGA}(\mg)=\Bider(\mg)$, for every Leibniz algebra $\mg$.
\end{itemize}

 In the next section we provide such description in the case of Poisson algebras.
 
 \bigskip
    
\section{Poisson Algebras} \label{sec:Poisson}

\bigskip

The main goal of this section is to study the representability of actions of the categories $\Pois_{\bF}$ of Poisson algebras and $\CPoisAlg_{\bF}$ of commutative Poisson algebra over a field $\bF$ with $\operatorname{char}(\bF) \neq 2$.

\begin{definition}
	A \emph{Poisson algebra} over $\mathbb{F}$ is a vector space $P$ over $\mathbb{F}$ endowed with two bilinear maps
	$$
	\cdot\colon P \times P \rightarrow P
	$$
	$$
	[-,-]\colon P \times P \rightarrow P
	$$
	such that $(P,\cdot)$ is an associative algebra, $(P,[-,-])$ is a Lie algebra and the \emph{Poisson identity} holds:
	$$
	[p,qt]=[p,q]t+q[p,t], \; \; \forall p,q,t \in P,
	$$
	i.e.\ the adjoint map $[p,-]\colon P \rightarrow P$ is a derivation of the associative algebra $(P,\cdot)$. A Poisson algebra $P$ is said to be commutative if $(P, \cdot)$ is a commutative associative algebra.
\end{definition}

Now we recall the main properties of split extension of Poisson algebras. 

\begin{definition}\label{Rem:SplitPois}
	Let 
	\begin{equation}\label{diag:SplPois}
		\begin{tikzcd}
			0\ar[r]
			&V \arrow [r, "i"]
			&\hat{P} \arrow[r, shift left, "\pi"] & 
			P \ar[r]\ar[l, shift left, "s"]
			&0 
		\end{tikzcd}	
	\end{equation}
	be a split extension of Poisson algebras. The triple of bilinear maps 
	\[
	l \colon P \times V \rightarrow V, \; \; \; r\colon V \times P \rightarrow V,\; \; \; [\![-,-]\!]\colon P \times V \rightarrow V
	\]
	defined by
	$$
	p \ast y = s(p) \cdot_{\hat{P}} i(y), \; \; 
	x \ast q = i(x) \cdot_{\hat{P}} s(q), \; \;
	[\![p,y]\!]=[s(p),i(x)]_{\hat{P}}, \; \; \forall p,q \in P, \forall x,y \in V,
	$$
	where $p \ast -=l(p,-)$ and $- \ast q=r(-,q)$, is called the \emph{derived action} of $P$ on $V$ associated with \eqref{diag:SplPois}.
\end{definition}

As in the case of Leibniz algebras, given a triple of bilinear maps 
\[
l \colon P \times V \rightarrow V, \; \; \; r\colon V \times P \rightarrow V, \; \; \; [\![-,-]\!]\colon P \times V \rightarrow V,
\]
one can define two bilinear operations on $P \oplus V$
$$
(p,x)\diamond(q,y)=(pq,x\cdot_V y + p \ast y + x \ast q)
$$
and
$$
\{(p,x),(q,y)\}=([p,q],[x,y]_V+[\![p,y]\!]-[\![q,x]\!]),
$$
for every $(p,x),(q,y) \in P \oplus V$, and this defines a Poisson algebra structure on the vector space $P \oplus V$ if and only if the triple $(l,r,[\![-,-]\!])$ is a derived action of $P$ on $V$. 

This is equivalent to a set of conditions on $(l,r,[\![-,-]\!])$, as explained in the following proposition (again, see Theorem 2.4 in \cite{Orzech} and Proposition 1.1 in \cite{Datuashvili}).

\begin{prop}\label{PropPois}
	$(P \oplus V, \diamond, \{-,-\})$ is a Poisson algebra if and only if 
	\begin{enumerate}
		\item[(P1)] $(P \oplus V, \diamond)$ is an associative algebra, i.e.\ the following equalities hold
		\begin{itemize}
			\item $p \ast (x \cdot_V y)=(p \ast x)\cdot_V y$;
			\item $(x \cdot_V y) \ast p = x \cdot_V (y \ast p)$;
			\item $x \cdot_V (p \ast y) = (x \ast p) \cdot_V y$;
			\item $(p \ast x) \ast q = p \ast (x \ast q)$;
			\item $(pq)\ast x = p \ast (q \ast x)$;
			\item $x \ast (pq)= (x \ast p) \ast q$;
		\end{itemize}
	\item[(P2)] $(P \oplus V, \{-,-\})$ is a Lie algebra, i.e.
	\begin{itemize}
		\item $[\![p,[x,y]_V]\!]=[[\![p,x]\!],y]_V+[x,[\![p,y]\!]]_V$;
		\item $[\![[p,q],x]\!]=[\![p,[\![q,x]\!]]\!] - [\![q,[\![p,x]\!]]\!]$;
	\end{itemize}
	\item[(P3)]  $[\![pq,x]\!]=p\ast [\![q,x]\!] + [\![p,x]\!]\ast q$;
	\item[(P4)]  $[p,q] \ast x = p \ast [\![q,x]\!] - [\![q,p \ast x]\!]$;
	\item[(P5)]  $x \ast [p,q]=[\![q,x]\!] \ast p - [\![q,x \ast p]\!]$;
	\item[(P6)]  $p \ast [x,y]_V=[p \ast x, y]_V - [\![p,y]\!] \cdot_V x$;
	\item[(P7)]  $[x,y]_V \ast p = [x \ast p,y]_V - x \cdot_V [\![p,y]\!]$;
	\item[(P8)]  $[\![p,x \cdot_V y]\!] = [\![p,x]\!] \cdot_V y + x \cdot_V [\![p,y]\!]$;
\end{enumerate}
	for every $p,q \in P$ and for every $x,y \in V$. The resulting Poisson algebra is the \emph{semi-direct product} of $P$ and $V$ and it is denoted by $P \ltimes V$.
\end{prop}

\begin{rem}
	We recall that, for any split extension \eqref{diag:SplPois}, we have an isomorphism of split extensions
	\begin{equation*} 
		\begin{tikzcd} 
			0\ar[r] 
			&V \arrow [r, "i_2"] \ar[d, "\id_{V}"'] 
			&P  \ltimes V \arrow[r, shift left, "\pi_1"] \ar[d, "\theta"]& 
			P \ar[r]\ar[l, shift left, "i_1"]\ar[d, "\id_{P}"] 
			&0 \\
			0\ar[r] 
			&V \arrow [r, "i"] 
			&\hat{P} \arrow[r, shift left, "\pi"] & 
			P \ar[r]\ar[l, shift left, "s"] 
			&0 
		\end{tikzcd} 
	\end{equation*}
	where $i_1,i_2,\pi_1$ are the canonical injections and projection and $\theta \colon P \ltimes V \rightarrow \hat{P}$ is defined by $\theta(p,x)=s(p)+i(x)$, for every $(p,x) \in P \oplus V$.
\end{rem}

The category $\Pois_{\bF}$ has two obvious forgetful functors to the categories $\Ass_{\bF}$ and $\LieAlg_{\bF}$. Now, the category of Lie algebras is action representable: any split extension of a Lie algebra $P$ by another Lie algebra $V$ corresponds to a Lie algebra morphism \hbox{$\varphi\colon P \rightarrow \Der(V)$}. On the other hand, we know that $\Ass_{\bF}$ is a weakly action representable category and a split extension of an associative algebra $P$ by another associative algebra $V$ corresponds to an associative algebra morphism \hbox{$\varphi\colon P \rightarrow \Bim(V)$}. Notice that $\Der(V)$ is an actor, while $\Bim(V)$ is only a weak actor (see \Cref*{sec:Preliminaries}), in fact they are both universal strict general actors in the sense of \cite{Casas}. It is not clear whether the category $\Pois_{\bF}$ is weakly action representable, therefore in this section we start by describing a universal strict general actor $\operatorname{USGA}(V)$, when $V$ is a Poisson algebra. As explained in \Cref*{sec:Interest}, in general $\operatorname{USGA}(V)$ lies in a larger category $\cC_G$, which in this case is the category $\NAlg_{\bF}^2$ of algebras over $\bF$ with two not necessarily associative bilinear operations. Thus we look for a suitable subspace
$$
[V] \leq \Bim(V) \times\Der(V)
$$
and this must be endowed with two bilinear operations
$$
\cdot_{[V]},[-,-]_{[V]} \colon [V] \times [V] \rightarrow [V]
$$
 such that we can associate with every split extension of $P$ by $V$ in $\Pois_{\bF}$ a morphism
	$$
	\phi\colon P \rightarrow [V]
	$$
	 in $\NAlg_{\bF}^2$, defined by
	$$
	\phi(p)=(p \ast -, - \ast p, [\![p,-]\!]), \; \; \forall p \in P.
	$$
	Thus
	$$
	\phi(pq)=\phi(p) \cdot_{[V]} \phi(q)
	$$
	and
	$$
	\phi([p,q])=[\phi(p),\phi(q)]_{[V]}.
	$$
	In other words, by using  \Cref{PropPois}, the operations in $[V]$ must satisfy the following two conditions
	\begin{itemize}
		\item $(p \ast -, - \ast p, [\![p,-]\!]) \cdot_{[V]} (q \ast -, - \ast q, [\![q,-]\!])=\\=((pq) \ast -, - \ast (pq), p \ast [\![q,-]\!] + [\![p,-]\!] \ast q)$
		\item  $[(p \ast -, - \ast p, [\![p,-]\!]),(q \ast -, - \ast q, [\![q,-]\!])]_{[V]}=\\=(p \ast [\![q,-]\!] - [\![q,p \ast -]\!], [\![q,-]\!] \ast p - [\![q,- \ast p]\!], [\![p,[\![q,-]\!]]\!]-[\![q,[\![p,-]\!]]\!])$
	\end{itemize}
for every $p,q \in P$.
\\ \\
We define  $[V]$ as the subspace of all triples $(f,F,d)$ of $\Bim(V) \times \Der(V)$ satisfying the following set of equations:
\begin{enumerate}
\item[(V1)]  $f([x,y]_V)=[f(x),y]_V-d(y)\cdot_V x$;
\item[(V2)]  $F([x,y]_V)=[F(x),y]_V-x \cdot_V d(y)$;
\item[(V3)]  $d(x \cdot_V y) = d(x) \cdot_V y + x \cdot_V d(y)$;
\end{enumerate} 
for every $x,y \in V$. 

\begin{rem}
The subspace $[V]$ is not empty, since 
$$
(x \cdot_V -, - \cdot_V x, [x,-]_V) \in [V]
$$
for every $x \in V$. This triples are called \emph{inner multipliers of $V$}.
\end{rem}

Now we are ready to enunciate and prove the following.

\begin{thm}\label{thmPois}
Let $(V,\cdot_V,[-,-]_V)$ be a Poisson algebra.
\begin{itemize}
	\item[(i)] The space $[V]$ with the bilinear operations
	$$
	(f,F,d) \cdot_{[V]} (f',F',d')=(f \circ f', F' \circ F, f \circ d' + F' \circ d)
	$$
	$$
	[(f,F,d),(f',F',d')]_{[V]}=(f \circ d'-d' \circ f, F \circ d'-d' \circ F, d \circ d' - d' \circ d)
	$$
	is an object of $\NAlg_{\bF}^2$;
	\item[(ii)] The set $\Inn(V)$ of all inner multipliers of $V$ is a subalgebra of $[V]$ and it is a Poisson algebra itself;
	\item[(iii)] For every object $(P,\cdot,[-,-])$ in $\Pois_{\bF}$, the set of isomorphism classes of split extension of $P$ by $V$ are in bijection with the morphisms 
	$$
		\phi=(\phi_1,\phi_2,\phi_3)\colon U(P) \rightarrow [V]
	$$
	in $\NAlg_{\bF}^2$, where $U \colon \Pois_{\bF} \rightarrow \NAlg_{\bF}^2$ denotes the forgetful functor, such that $(\phi_1,\phi_2)\colon (P, \cdot) \rightarrow \Bim(V)$ is an acting morphism in the category $\Ass_{\bF}$.
  \item[(iv)] There exists a monomorphism of functors
  $$
  \tau\colon \SplExt(-,V) \rightarrowtail \Hom_{\NAlg_{\bF}^2}(U(-),[V]),
  $$
 such that an arrow $(\phi \colon U(P) \rightarrow [V]) \in \Imm(\tau_P)$
  if and only if $(\phi_1,\phi_2) $ is an acting morphism in $\Ass_{\bF}$.
  \item[(v)] If $([V], \cdot_{[V]}, [-,-]_{[V]})$ is a Poisson algebra, then the pair $([V],\tau)$ becomes a weak representation for the functor $\SplExt(-,V)$.
\end{itemize}
\end{thm}

\begin{proof}{\ }\\
\begin{itemize}
\item[(i)] In order to show that $[V]$ is an object of $\NAlg_{\bF}^2$, we have to prove that the bilinear operations are well defined. We observe that
	$$
	(f \circ d' - d' \circ f, F \circ d' - d' \circ F)
 \in \Bim(V)	$$
  and
  $$
  f \circ d' + F' \circ d \in \Der(V),
  $$
	for every $(f,F,d),(f',F',d') \in [V]$. This follows from equations (V1)-(V2)-(V3), since
	$$
	(f \circ d' - d' \circ f)(x \cdot_V y)=(f \circ d' - d' \circ f)(x) \cdot_V y,
	$$
	$$
	(F \circ d' - d' \circ F)(x \cdot_V y)=x \cdot_V (F \circ d' - d' \circ F)(y),
	$$
	$$
	x \cdot_V (f \circ d' - d' \circ f)(y)=(F \circ d' - d' \circ F)(x) \cdot_V y
	$$
	and
	$$
	(f \circ d' + F' \circ d)([x,y]_V)=
	$$
	$$
	=[(f \circ d' + F' \circ d)(x),y]_V+[x,(f \circ d' + F' \circ d)(y)]_V,
	$$
	for every $x,y \in V$. Moreover the resulting triples 
	$$
	(f \circ f', F' \circ F, f \circ d' + F' \circ d)
	$$
	$$
	(f \circ d'-d' \circ f, F \circ d'-d' \circ F, d \circ d' - d' \circ d)
	$$
	belong to $[V]$, i.e.\ they satisfy equations (V1)-(V2)-(V3). Here we show this statement only for the second triple, since for the first triple the computations are similar. We have that
	\begin{equation*}
		(f \circ d' - d' \circ f)[x,y]_V=
	\end{equation*}
\begin{equation*}
	=f([d'(x),y]_V+[x,d'(y)]_V)-d'([f(x),y]_V-d(y)\cdot_V x)=
\end{equation*}
\begin{equation*}
	=[f(d'(x)),y]_V-d(d'(y)) \cdot_V x - [d'(f(x)),y]_V+d'(d(y)) \cdot_V x=
\end{equation*}
\begin{equation*}
	=[(F \circ d' - d' \circ F)(x),y]_V-(d \circ d'-d' \circ d)(y) \cdot_V x.
\end{equation*}
In the same way one can check that
\begin{equation*}
	(F \circ d' - d' \circ F)[x,y]_V=[(F \circ d' - d' \circ F)(x),y]_V-x \cdot_V (d \circ d' -d' \circ d)(y).
\end{equation*}
Finally
\begin{equation*}
(d \circ d' -d' \circ d)(x \cdot_V y)=
\end{equation*}
\begin{equation*}
	=d(d'(x) \cdot_V y +x \cdot_V d'(y)) - d'(d(x) \cdot_V y +x \cdot_V d(y))=
\end{equation*}
\begin{equation*}
	=d(d'(x))\cdot_V y + x \cdot_V d(d'(y)) - d'(d(x)) \cdot_V y - x \cdot_V d'(d(y)) =
\end{equation*}
\begin{equation*}
	= (d \circ d' -d' \circ d)(x) \cdot_V y + x \cdot_V (d \circ d' -d' \circ d)(y).
\end{equation*}
Thus $[V]$ is an object of $\NAlg_{\bF}^2$.
    \item[(ii)] The subspace $\Inn(V)$ is precisely the image of the morphism 
    $$
    \Inn \colon V \rightarrow [V]
    $$
    defined by
    $$
    x \mapsto (x \cdot_V -, - \cdot_V x, [x,-]_V), \;\; \forall x \in V. 
    $$
	\item[(iii)] We associate with any split extension. 
	\begin{equation*}
		\begin{tikzcd}
			0\ar[r]
			&V \arrow [r, "i"]
			&\hat{P} \arrow[r, shift left, "\pi"] & 
			P \ar[r]\ar[l, shift left, "s"]
			&0 
		\end{tikzcd}	
	\end{equation*}
	in the category $\Pois_{\bF}$ the morphism
	$$
	U(P) \rightarrow [V]
	$$
	in $\NAlg_{\bF}^2$, where $U \colon \Pois_{\bF} \rightarrow \NAlg_{\bF}^2$ denotes the forgetful functor, defined by
	$$
	p \rightarrow (p \ast -, - \ast p, [\![p,-]\!]), \; \; \forall p \in P,
	$$
	where the bimultiplier $(p \ast -, - \ast p)$ and the derivation $[\![p,-]\!]$ are as in \Cref{Rem:SplitPois}. Since $\hat{P}$ is also a split extension of $(P,\cdot)$ by $(V, \cdot_V)$ in the category $\Ass_{\bF}$, we have that
	$$
	p \ast (x \ast q) = (p \ast x) \ast q,
	$$
	for every $p,q \in P$ and $x \in V$. Conversely, given a Poisson algebra $P$ and a morphism $\phi=(\phi_1,\phi_2,\phi_3) \in \Hom_{\textbf{NAlg}^2_{\bF}}(U(P),[V])$ defined by
	$$
	\phi(p)=(p*_{\phi}-,-*_{\phi}p,[\![p,-]\!]_{\phi}), \; \; \forall p \in P,
	$$
	such that $(\phi_1,\phi_2)\colon (P,\cdot) \rightarrow \Bim(V)$ is an acting morphism in $\Ass_{\bF}$, we can associate with $\phi$ the split extension of Poisson algebras
	\begin{equation*}
		\begin{tikzcd}
			0\ar[r]
			&V \arrow [r, "i"]
			&(P \oplus V, \diamond_{(\phi_1,\phi_2)},\{-,-\}_{\phi_3}) \arrow[r, shift left, "\pi"] & 
			P \ar[r]\ar[l, shift left, "s"]
			&0 
		\end{tikzcd}	
	\end{equation*}
	where
	$$
	(p,x) \diamond_{(\phi_1,\phi_2)} (q,y) = (pq, x \cdot_V y + p*_{\phi}y + x*_{\phi}q)
	$$
	and
	$$
	\{(p,x),(q,y)\}_{\phi_3}=([p,q],[x,y]_V + [\![p,y]\!]_{\phi}-[\![q,x]\!]_{\phi}),
	$$
	for every $(p,x),(q,y) \in P \oplus V$. One can check that these bilinear operations define a Poisson algebra structure on $P \oplus V$.
\item[(iv)] We define
	$$
	\tau\colon \SplExt(-,V) \rightarrowtail \Hom_{\NAlg_{\bF}^2}(U(-),[V])
	$$
	in the following way: for every object $P$ in  $\Pois_{\bF}$, $\tau_P$ associates with any split extensions of $P$ by $V$ the morphism
	$$
	U(P) \rightarrow [V]
	$$
    defined as in (iii). By the description of split extensions in \Cref{Rem:SplitPois}, each component $\tau_P$ is injective since every morphism which belongs to $\Imm(\tau_P)$ determines a unique split extension of $P$ by $V$. One can check that the family of injections 
    $$
    \tau_P\colon \SplExt(P,V) \rightarrowtail \Hom_{\NAlg_{\bF}^2}(U(P),[V])
    $$
    is natural in $P$. By (iii), an arrow $\phi=(\phi_1,\phi_2,\phi_3) \in \Hom_{\NAlg_{\bF}^2}(U(P),[V])$ belongs to $\Imm(\tau_P)$ if and only if $(\phi_1,\phi_2) \colon (P,\cdot) \rightarrow \Bim(V)$ is an acting morphism in $\Ass_\bF$.
    \item[(v)] The last statement follows from \Cref{USGA}, since $[V]=\operatorname{USGA}(V)$.
\end{itemize}
\end{proof}

The following example shows that $([V], \cdot_{[V]}, [-,-]_{[V]})$ is not in general a Poisson algebra.

\begin{ex}
	Let $V=\bF^2$ be the the abelian two-dimensional algebra (i.e. $x \cdot_V y=[x,y]_V=0$, for every $x,y \in V$). It turns out that
	$$
	[V]=\operatorname{End}(V)^3 \cong \operatorname{M}_2(\bF)^3,
	$$
	as vector spaces, since every linear endomorphism of $V$ is represented by a $2 \times 2$ matrix with respect to a fixed basis. Then the bilinear operations of $[V]$ can be represented as
	$$
	(A,B,C) \cdot_{[V]} (A',B',C')=(AA', B'B, AC'+B'C),
	$$
	$$
	[(A,B,C),(A',B',C')]_{[V]}=(AC'-C'A, BC'-C'B, CC'-C'C),
	$$
	for every $(A,B,C),(A',B',C') \in \operatorname{M}_2(\bF)^3$ and one can check that $[V]$ is not a Poisson algebra since, for instance, the bracket $[-,-]_{[V]}$ is not skew-symmetric.
\end{ex}

By Theorem 3.9 of \cite{Casas}, we can deduce that the category $\Pois_{\bF}$ is not action representable. Indeed, since for a Poisson algebra $V$, $\operatorname{USGA}(V)$ is not in general a Poisson algebra, then $V$ does not admit an actor.

The following remark shows that there are special cases where $\tau$ becomes a natural isomorphism.

\begin{rem}
	Let $(V,\cdot_V,[-,-]_V)$ be a Poisson algebra such that the annihilator 
	$$\operatorname{Ann}(V)=\lbrace x \in V \; | \; x \cdot_V y = y \cdot_V x = 0, \; \forall y \in V \rbrace$$
	of the associative algebra $(V,\cdot_V)$ is trivial or $(V^2,\cdot_V)=(V,\cdot_V)$. In this case we have that
	\begin{equation}\label{eqPois}
		f \circ F' = F' \circ f,
	\end{equation}
	for every $(f,F),(f',F') \in \Bim(V)$ (see \cite{Casas} for more details). It follows that, for any other Poisson algebra $P$, every arrow
	$$
	\phi \colon U(P) \rightarrow [V]
	$$
	belongs to $\Imm(\tau_P)$ and we have a natural isomorphism
	$$
	\SplExt(-,V) \cong \Hom_{\NAlg_{\bF}^2}(U(-),[V]).
	$$
\end{rem}

Notice that the conditions $\operatorname{Ann}(V)=0$ and $V^2=V$ are not necessary to obtain equation \eqref{eqPois}. For instance, if $V=\bF$ is the abelian one-dimensional algebra, then $\operatorname{Ann}(V)=V$, $V^2=0$, $[V] \cong \bF^3$ as vector spaces (every linear endomorphism of $V$ is of the form $\varphi_{a} \colon x \mapsto a x$, with $a\in \bF$) and every left multiplier of $V$ commutes with every right multiplier. Moreover it turns out that
$$
(\varphi_{a},\varphi_{b},\varphi_{c}) \cdot_{[V]} (\varphi_{a'},\varphi_{b'},\varphi_{c'}) = (\varphi_{aa'},\varphi_{b'b},\varphi_{ac'+b'c}) 
$$
is an associative product and
$$
[(\varphi_{a},\varphi_{b},\varphi_{c}), (\varphi_{a'},\varphi_{b'},\varphi_{c'})]_{[V]}= (0,0,0).
$$
Thus $[V]$ is a Poisson algebra and 
$$
\SplExt(-,V) \cong \Hom_{\Pois_{\bF}}(-,[V]),
$$
i.e., $[V]$ is the actor of $V$.

\subsection*{Commutative Poisson algebras}

We want now to describe the universal strict general actor $[V]_c=\operatorname{USGA}(V)$ of an object $V$ of the subvariety $\CPoisAlg_{\bF}$ of commutative Poisson algebras over $\bF$.

If $V$ is a commutative Poisson algebra, then we define $[V]_c$ as the algebra of all pairs $(f,d) \in \operatorname{M}(V) \times \Der(V)$, where
$$
\operatorname{M}(V)=\lbrace f \in \End(V) \; | \; f(xy)=f(x)y, \; \forall x,y \in V \rbrace
$$
is the associative algebra of \emph{multipliers} of $V$, such that 
\begin{enumerate}
	\item[(V1)]  $f([x,y]_V)=[f(x),y]_V-d(y)\cdot_V x$;
	\item[(V2)]  $d(x \cdot_V y) = d(x) \cdot_V y + x \cdot_V d(y)$;
\end{enumerate} 
endowed with the two bilinear operations
$$
(f,d) \cdot_{[V]_c} (f',d')=(f \circ f', f \circ d' + f' \circ d),
$$
$$
[(f,d),(f',d')]_{[V]_c}=(f \circ d'-d' \circ f, d \circ d' - d' \circ d),
$$
for every $(f,d),(f',d') \in [V]_c$. One can check that $[V]_c$ is isomorphic to the subalgebra of $[V]$ of triples of the form $(f,f,d)$.

Using the notation of \Cref*{thmPois}, one can associate, with any split extension of $P$ by $V$ in $\CPoisAlg_{\bF}$, a morphism
$$
\phi \colon \tilde{U}(P) \rightarrow [V]_c, \; \; p \mapsto (p \ast -, [\![p,-]\!]), \; \; \forall p \in P
$$
in $\NAlg_{\bF}^2$, where $\tilde{U} \colon \CPoisAlg_{\bF} \rightarrow \NAlg_{\bF}^2$ denotes the forgetful functor. Conversely, if $P$ and $V$ are commutative Poisson algebras, every morphism $\phi \colon \tilde{U}(P) \rightarrow [V]_c$ in $\NAlg_{\bF}^2$ defines a commutative Poisson algebra split extension. Indeed, by (iii) of \Cref{thmPois}, such $\phi \in \Imm(\tau_P)$ if and only if $p \mapsto p \ast -$ defines an action in the category $\CAAlg_{\bF}$ of commutative associative algebra over $\bF$, and moreover 
$\Act_{\CAAlg_{\bF}}(-,V) \cong \Hom_{\Ass_{\bF}}(\overline{U}(-),\operatorname{M}(V))$ (see \cite{ActRepr}), where $\overline{U} \colon \CAAlg_{\bF} \rightarrow \Ass_{\bF}$ is the forgetful functor. Thus there exists a natural isomorphism 
$$
\SplExt(-,V) \cong \Hom_{\NAlg_{\bF}^2}(\tilde{U}(-),[V]_c).
$$

Finally, we observe that also in this case, $[V]_c$ needs not be an object of $\CPoisAlg_\bF$. For instance, if $V=\bF^2$ is the abelian two-dimensional algebra, then 
\[
[V]_c=\operatorname{M}(V) \times \Der(V)=\End(V)^2
\]
as a vector space, and it is easy to check that the bilinear operation
\[
(f,d) \cdot_{[V]_c} (f',d')=(f \circ f', f \circ d' + f' \circ d)
\]
is not commutative, nor associative.
%
%

\section{Open Problem}
We observe that our investigation does not clarify whether the categories $\Pois_{\bF}$ and $\CPoisAlg_{\bF}$ of Poisson algebras and commutative Poisson algebras are weakly action representable or not. Nevertheless, the explicit construction of the universal strict general actor $\operatorname{USGA}(V)$ provides an operational tool to represent actions on/split extensions by a (commutative) Poisson algebra~$V$. Notice that this construction was generalized for a general variety of non-associative algebras over a field in~\cite{WRAAlg}, where the authors proved that the category of commutative associative algebras is weakly action representable, even though the USGA is not an object of it (see~\cite[Theorem 2.11]{WRAAlg}). 

We further observe that the first examples of action accessible varieties of non-associative algebras that fail to be weakly action representable were found in~\cite{WRAAlg2}, where it was proved that the categories of $n$-solvable Lie algebras ($n \geq 2$) and $k$-nilpotent Lie algebras ($k \geq 3$) do not satisfy weak action representability. 
%
%

\bigskip
	
%
%

\printbibliography

@article{WAR,
author = {Janelidze, G.},
title = {Central extensions of associative algebras and weakly action representable categories},
journal = {Theory and Applications of Categories},
volume = {38},
number = {36},
pages = {1395-1408},
year = {2022},
}

@article{WRAAlg,
author = {Brox, J. and García-Martínez, X. and Mancini, M. and Van der Linden, T. and Vienne, C},
title = {Weak representability of actions of non-associative algebras},
journal = {Journal of Algebra},
year = {2025},
volume = {669},
number = {18},
pages = {401-444},
doi = {https://doi.org/10.1016/j.jalgebra.2025.02.007},
}

@article{WRAAlg2,
author = {García-Martínez, X. and Mancini, M.},
title = {Action accessible and weakly action representable varieties of algebras},
year = {2025},
pages = {submitted, preprint available at \texttt{arXiv:2503.17326}.},
}

@article{Datuashvili,
author = {Datuashvili, T.},
title = {Cohomologically trivial internal categories in categories of groups with operations},
journal = {Applied Categorical Structures},
volume = {3},
pages = {221-237},
year = {1995},
doi = {https://doi.org/10.1016/j.jalgebra.2025.02.007},
}

@article{Gray,
	author = {Gray, J. R. A.},
	title = {A note on the relationship between action accessible and weakly action representable categories},
	journal = {Theory and Applications of Categories},
	year = {2025},
	volume = {44},
	number = {8},
	pages = {272-276},

@article{MM,
	author = {Mancini, M.},
	title = {Biderivations of low-dimensional {L}eibniz algebras},
	journal = {H. Albuquerque, J. Brox, C. Martínez, P. Saraiva (eds.), Non-Associative Algebras and Related Topics. NAART 2020. Springer Proceedings in Mathematics \& Statistics},
	year = {2023},
	volume = {427},
	number = {8},
	pages = {127-136},
	publisher = {Springer Cham},
	doi = {https://doi.org/10.1007/978-3-031-32707-0_8},
}

@article{Cigoli,
title = {Action accessibility via centralizers},
journal = {Journal of Pure and Applied Algebra},
author = {Cigoli, A. S. and Mantovani, S.},
volume = {216},
number = {8-9},
pages = {1852-1865},
year = {2012},
doi = {https://doi.org/10.1016/j.jpaa.2012.02.023},
}

@article{Orzech,
author = {Orzech, G.},
title = {Obstruction theory in algebraic categories, I},
journal = {Journal of Pure and Applied Algebra},
volume = {2},
number = {4},
pages = {287-314},
year  = {1972},
doi= {https://doi.org/10.1016/0022-4049(72)90009-6}
}

@article{Metere,
author = {Cigoli, A. S. and Metere, G. and Montoli, A.},
title = {Obstruction theory in action accessible categories},
journal = {Journal of Algebra},
volume = {385},
number = {3},
pages = {27-46},
year  = {2013},
doi = {https://doi.org/10.1016/j.jalgebra.2013.03.020},
}

@article{Montoli,
author = {Montoli, A.},
title = {Action accesibility for categories of interest},
journal = {Theory and Applications of Categories},
volume = {23},
number = {1},
pages = {7-21},
year  = {2010},
}

@article{Tim,
author = {García-Martínez, X. and Tsishyn, M. and Van der Linden, T. and Vienne, C.},
title = {Algebras with representable representations},
journal = {Proceedings of the Edinburgh Mathematical Society},
volume = {64},
number = {2},
pages = {555-573},
year  = {2021},
doi={https://doi.org/10.1017/S0013091521000304},
}

@article{ActRepr,
author = {Borceux, F. and Janelidze, G. and Kelly, A.},
year = {2005},
number = {1},
pages = {244-286},
title = {On the representability of actions in a semi-abelian category},
volume = {14},
journal = {Theory and Applications of Categories}
}

@article{IntAct,
	author = {Borceux, F. and Janelidze, G. and Kelly, A.},
	year = {2005},
	number = {2},
	pages = {235-255},
	title = {Internal object actions},
	volume = {46},
	journal = {Commentationes Mathematicae Universitatis Carolinae}
}

@article{Casas,
author = {Casas, J. M. and Datuashvili, T. and Ladra, M.},
title = {Universal Strict General Actors and Actors in Categories of Interest},
journal = {Applied Categorical Structures},
volume = {18},
pages = {85-114},
year  = {2010},
doi = {https://doi.org/10.1007/s10485-008-9166-z},
}

@article{MacLane58,
author = {Mac Lane, S.},
title = {Extensions and obstructions for rings},
volume = {2},
journal = {Illinois Journal of Mathematics},
number = {3},
pages = {316-345},
year = {1958},
doi = {https://doi.org/10.1215/ijm/1255454537}
}

@article{loday1993version,
	title={Une version non commutative des algebres de Lie: les algebres de Leibniz},
	author={Loday, J.-L.},
	journal={L'Enseignement Math\'ematique},
	volume={39},
	number = {3-4},
	pages={269-293},
	year={1993}
}

@article{Semi-Ab,
title = {Semi-abelian categories},
journal = {Journal of Pure and Applied Algebra},
volume = {168},
number = {2},
pages = {367-386},
year = {2002},
doi = {https://doi.org/10.1016/S0022-4049(01)00103-7},
author = {Janelidze, G. and Márki, L. and Tholen, W.},
}

@article{act_accessible,
     author = {Bourn, D. and Janelidze, G.},
     title = {Centralizers in action accessible categories},
     journal = {Cahiers de Topologie et G\'eom\'etrie Diff\'erentielle Cat\'egoriques},
     volume = {50},
     number = {3},
     year = {2009},
     pages = {211-232},
}

\end{document}